\documentclass[,12pt]{elsarticle}
\journal{Journal of Symbolic Computation}
\usepackage{amsfonts}
\usepackage{amssymb}
\usepackage{graphicx}
\usepackage[usenames]{color}
\usepackage[all]{xy} %para diagramas conmutativos

\usepackage{amsthm}

\usepackage[active]{srcltx}
\usepackage[algo2e,ruled,vlined]{algorithm2e}

\usepackage{color}
\def\g2{ \hbox{\got g}_2}
\def\f4{\hbox{\got f}_4}
\def\d4{\hbox{\got d}_4}
\def\e6{\hbox{\got e}_6}
\def\fs4{\hbox{\gots f}_4}
\def\F4{\hbox{\got F}_4}
\def\Fs4{\hbox{\gots F}_4}

\def\es7{ \hbox{\got e}_7}
\def\es8{ \hbox{\got e}_8}

\def\F{\mathbb F}

\def\al{\ifcase\xypolynode\or F \or A\or B\or C\or D\or G\fi}
\def\ala{\ifcase\xypolynode\or a \or b\or c\or d\or g\or f\fi}

\newtheorem{te}{Theorem}

\newtheorem{defi}{Definition}
\newtheorem{ex}{Example}

\usepackage{graphicx}
\begin{document}
%%%%%%%%%%%%%%%%%%%%%%%%%%%
\begin{frontmatter}

\title{An algorithm to compute  minimal Sullivan algebras}
\author{Antonio Garv{\' i}n}
\address{Applied Math Dept., University of Malaga, Malaga, Spain\\
email: garvin@uma.es}
\author{Rocio Gonzalez-Diaz}
\address{Applied Math Dept. I, University of Seville, Seville, Spain\\
email: rogodi@us.es}
\author{Belen Medrano}
\address{Applied Math Dept. I, University of Seville, Seville, Spain\\
email: belenmg@us.es}
%%%%%%%%%%%%%%%%%%%%%%%%%%%

\begin{abstract}
In this note, we give an algorithm that starting with a Sullivan algebra gives us its minimal model. This algorithm is a kind of modified  AT-model algorithm used to compute in the past other kinds of topology information such as (co)homology, cup products on cohomology and persistent homology. Taking as input a (non-minimal) Sullivan algebra $A$ with an ordered finite set of generators preserving the filtration defined on  $A$, we obtain as output a minimal Sullivan algebra with the same rational cohomology as $A$.
\end{abstract}

\begin{keyword}
Sullivan algebras, minimal models, chain homotopy, AT-model 
\end{keyword}

%%%%%%%%%%%%%%%%%%%%%%%%%%%%%%

\end{frontmatter}
 %%%%%%%%%%%%%%%%%%%%%%%%%%%%%%  

\section{Introduction}

 Algebraic Topology consists, essentially, in the study of algebraic invariants associated with topological spaces. 
 One of these invariants is the homotopy type of a space.
At present we are still far from having a complete algebraic description of homotopy. This could be the reason why the computability of the homotopy type is still an open problem nowadays.  However, this description exists within the framework of Rational Homotopy Theory (\cite{GTM2000,felx} is nowadays a standard reference for the theory).
Two distinct approaches to Rational Homotopy Theory were given independently by Quillen \cite{quillen} and Sullivan \cite{sullivan} at the end of the sixties. 
%In Rational Homotopy Theory via 
Following the classical Sullivan's approach \cite{Sul78}, we associate to every rational homotopy type of $1$-connected spaces (nilpotent spaces, in general) of finite type in a unique way (up to isomorphism) its minimal model, which is a commutative differential graded algebra.
This association is functorial and encode the homotopic properties of the space up to rationalization.
From the theoretical point of view, we can always get the minimal model of a space, but in practical examples, effective computations are not always possible.
To get the minimal model is essential because it determines the rational homotopy type of the space
and sometimes we are able to recognize the space from its minimal model.

A weaker version of a minimal model is the one of a Sullivan algebra.
The basic difference between both of them is that a minimal model is a Sullivan algebra such that its differential does not have linear terms.
There are many situations that we get a Sullivan algebra instead of a minimal model. For example, if we have a fibration, we can associate it with a minimal extension or relative Sullivan algebra and we can get from it a model of the total space. Starting with a minimal model of the base we can extend it to obtain a Sullivan algebra model of the total space of the fibration, but this model, in general, is not minimal.
If we want to try to recognize the total space we can follow a standard procedure:
As a Sullivan algebra is always the tensor product of a minimal Sullivan algebra and a contractible one,
the idea is to eliminate the contractible part.
In concrete examples, the procedure is ``ad hoc''. 
See, for example, \cite{Cirici}, where the existence of Sullivan minimal models of operads algebras is given by adapting Sullivan's step by step construction.
In general,
such ``ad hoc'' procedure is not effective
 and a computational approach would be of interest.

Our aim in this paper is to give an effective algorithm that starting with a non-minimal Sullivan algebra ends with its minimal model.
The idea is to use a modification of the incremental algorithm for computing the AT-model of a chain complex \cite{GR2005,GR2009,GR2009b}.
Essentially the incremental algorithm is a sequence of chain contractions of modules that starting with a chain complex that ends with its homology. The key idea is to get a sequence of chain contraction of algebras
instead of modules, not finishing with zero differential but with a differential with no linear terms.
We prove that this is possible starting with a Sullivan algebra with a finite number of generators and we give several examples of effective computation of minimal models. An implementation  of our method that runs in CoCalc (https://cocalc.com/) 
can be download to test from the website http://grupo.us.es/cimagroup/downloads.htm.

The paper is organized as follows. Section 1, 2 and 3 are devoted to introducing the background of the paper (DG-modules, AT-models, CDG-algebras and Sullivan algebras). In Section 4,  as a main result of the paper, an algorithm is provided to compute minimal Sullivan algebras from Sullivan algebras not necessarily being minimal. Examples are given in Section 5. The paper ends with a section devoted to conclusions and future work.

\section{Differential graded  modules and AT-models}

In this section, we introduced the background for DG-modules needed to understand the paper and an algorithm for computing  AT-models \cite{GR2005} which is a precursor of the algorithm for computing minimal Sullivan algebras, as we will see later.

Let $\Lambda$ be a commutative ring  with $1 \neq 0$, taken henceforth as ground ring. Denote by $M=\Lambda (m_0,\dots,m_n)$ the $\Lambda$-module $M$ finitely generated by the elements $m_0,\dots,m_n \in M$. That is, for any $x \in M$, there exist $\lambda_0,\dots, \lambda_n \in \Lambda$ such that $x = \lambda_0 m_0 + \dots + \lambda_n m_n$. For each $i$, $0\leq i\leq n$, we denote the coefficient $\lambda_i$ by coeff$(x,m_i)$. Finally, we say that $x\in M$ has index $j$ and write index$(x)=j$ if $x=m_j$.

\begin{defi}
A graded module (G-module) $M$ is a family of
$\Lambda$-vector spaces $\{M^n\}_{n \in \mathbb{Z}}$.
We will suppose that $M^n=0$ for $n < 0$. 
We say that an element  $x \in M$ is homogeneous if  $x \in M^n$ for some $n$ and, in such case, we say that $x$
has degree $n$ and write $|x|=n$.
\end{defi}

\begin{defi}
 A G-module morphism $f:M\to N$ of degree $q$ 
 (denoted by $|f|=q$) is a family  of  homomorphisms $\{f^n\}_{n \in \mathbb{Z}}$
  such that $f^n: M^n \to N^{n+q}$ for all $n$.
\end{defi}

\begin{defi}
Given two G-modules $M$ and $N$, $M\otimes N$ is defined as the G-module  $$M\otimes N=\oplus_{n\in\mathbb{Z}} 
(M\otimes N)^n$$
where 
$$(M\otimes N)^n:=\oplus_{p+q=n}(M_p\otimes N_q)$$ 
and the degree of an element $a \otimes b \in M \otimes N$ is
$$|a \otimes b| = |a| + |b|$$
\end{defi}

\begin{defi}
Given two G-module morphisms, $f:M\to N$ and $g:M\to N$, the tensor product $f \otimes g$ is defined adopting the Koszul's convention as:
$$(f\otimes g)(x\otimes y)=(-1)^{|g|\cdot |x|}f(x)\otimes g(y)$$
\end{defi}

\begin{defi}
A differential over a G-module $M$ is a G-module morphism $d_M:M\to M$ such that $|d_M|=\pm 1$. In this paper all the differentials will have degree $+1$. 
\end{defi}

\begin{defi}
A DG-module is a G-module $M$ endowed with a differential $d_M:M\to M$ which satisfies that $d_M d_M=0$. It is denoted by $(M,d_M)$.
\end{defi}

\begin{defi}
A DG-module morphism $f:(M,d_M)\to (N,d_N)$ of degree $q$ is a G-module morphism $f:M\to N$ of degree $q$ such that $d_N f = (-1)^q f d_M$.
\end{defi}

\begin{defi}
If $(M,d_M)$ is a DG-module, then ($M\otimes M,d_{M\otimes M}$) is a DG-module with the differential given by:
$$d_{M\otimes M} = id_M \otimes d_M + d_M\otimes id_M$$
\end{defi}

A contraction is a chain homotopy equivalence between two DG-modules.
\begin{defi}
Given two DG-modules $(M,d_M)$ and $(N,d_N)$, a contraction from $(M,d_M)$ to $(N,d_N)$ is a triplet $(f,g,\phi)$ such that:
\begin{itemize}
\item $f: (M,d_M)\to (N,d_N)$ and $g: (N,d_N)\to (M,d_M)$ are DG-module morphisms of degree $0$ satisfying that:
$$fg=id_N$$
\item $\phi: M\to M$ is a G-module morphism
of degree $-1$ satisfying that:
$$\begin{array}{c}
f\phi=0, \phi g=0, \phi\phi=0,\\
id_M-gf=\phi d_M+ d_M \phi.
\end{array}$$
\end{itemize}
\end{defi}

If a contraction between two DG-modules exists, then it is easy to see that both DG-modules has isomorphic (co)homology.

An AT-model for a DG-module is nothing more than a contraction from the DG-module to its homology.  

\begin{defi}
An AT-model  for a DG-module $(M,d_M)$ is a  contraction $(f,g,\phi)$ from $(M,d_M)$ to a finitely-generated DG-module $(H,d_H)$ with null-differential, that is, $d_H=0$.
\end{defi}

Therefore, an AT-model $(f,g,\phi)$ for $(M,d_M)$ satisfies that:
$$\begin{array}{c}
fd_M=d_M g=f \phi=\phi g=\phi\phi=0,\\
id_M-gf=\phi d_M+ d_M \phi,\;fg=id_H,\\
\phi d_M\phi=\phi,\;d_M\phi d_M = d_M.
\end{array}$$

 The following algorithm computes an AT-model  for $(M,d_M)$.
 In this incremental algorithm, we start with a filtering or order of the generators. The condition is that the differential of a generator is a linear combination of the generators that appear previously in the order.
  
\begin{algorithm2e}[H]
{\bf Input: } A finitely-generated DG-module $(M,d_M)$ with $M=\Lambda \langle m_0,\dots,m_n\rangle$ such that $d_M(m_i)\in \Lambda \langle m_0,\dots,m_{i-1}\rangle$ for $1\leq i\leq n$
  \BlankLine
{\bf  Initialize:} $H:=\{m_0\}$, $f(m_0):=m_0$, $g(m_0):=m_0$  and $\phi(m_0):=0$\\
\For{$i=1$ {\bf to} $n$}
	{Let $a= f d_M (m_i)$ and
	$b= m_i-\phi d_M (m_i)$ \\
	\If{$a=0$}
    		{
    		(a new homology class $\alpha_i$ is born when $m_i$ is added)\\
    		$H:=H\cup\{m_i\},\;f(m_i):=m_i$,
    	 	$g(m_i):=b$
    	 and	$\phi(m_i):=0$ 
		}
	\If{$a\neq 0$ 
	%$m_j\in f d_M (m_i)$ such that 
%	f d_M(m_i)\}$
}	
        {
        let  $j=\max\{$index$(m): m\in 
	H$ and coeff$(a,m)\neq 0\}$\\
	(the homology class  $\alpha_j$ dies when  $m_i$ is added)\\
        $H:=H_{i-1}\setminus\{m_j\},\;f(m_i):=0$ and 
        $\phi(m_i):=0$\\
        	\ForEach{$m\in \{m_0,\dots,m_{i-1}\}$}
		{$f(m):=f(m)-\lambda a$ and
		%$f(m):=f(m)-\lambda fd_M(m_i)$ 
			$\phi(m):=\phi(m)+\lambda b$
		 %$\phi(m):=\phi(m)+\lambda(m_i-\phi d_M(m_i))$
		 where $\lambda:=\frac{\mbox{\tiny coeff}(f(m),m_j)}{\mbox{\tiny coeff}(a,m_j)}$ 
		 %and $g(m):= m-\phi d_M(m)$
		}
		}
	}
	     \BlankLine
{\bf Output: } An AT-model $(f,g,\phi)$ for $(M,d_M)$.
  \caption{AT-model for computing homology \cite{GR2005}}.
   \label{alg}
\end{algorithm2e}

The complexity of Algorithm \ref{alg} is cubic in the number of generators of $M$ (see \cite{GR2005}). The (co)homology of $(M,d_M)$ is isomorphic to the one of $(H,d_H=0)$ since $(f,g,\phi)$ is a contraction.  Besides, Algorithm \ref{alg} can be used to  compute more sophisticated topology information such as  cup products on cohomology \cite{GR2014} or persistent homology \cite{GR2011}.

\section{Graded (differential) algebras}

In this section we recall the notion of commutative differential graded algebras.

\begin{defi} A G-algebra $(A,\mu_A)$ is a G-module $A$ together with    a G-module morphism 
$\mu_A: A\otimes A\to A$
of degree $0$
such that: $$\mu_A(\mu_A(x\otimes y)\otimes z) = \mu_A(x\otimes \mu_A(y\otimes z))\;\;\mbox{($\mu_A$ is associative)}$$ and 
$$\mu_A(1_A\otimes x) = x = \mu_A(x\otimes 1_A)
\; \;\mbox{($1_A\in A^0$ is an identity)}
.$$
\end{defi}

 \begin{defi}
 A CG-algebra $(A,\mu_A)$  is a G-algebra that is commutative in the graded sense, that is:
 $$\mu_A(a_p\otimes a_q)=(-1)^{pq} \mu_A(a_q\otimes a_p)\;\mbox{ where }a_p\in A^p
 \,\mbox{ and }\,a_q\in A^q.$$
  \end{defi}

\begin{defi}
A DG-algebra $(A,\mu_A,d_A)$ is a G-algebra $(A,\mu_A)$ together with a G-module morphism
$d_A: A\to A$
of degree $+1$ such that:
$$d_Ad_A=0 \;\;\mbox{($d_A$ is a differential)}$$
and
 $$d_A\mu_A(x\otimes y) = \mu_A(d_A(x)\otimes y) + (-1)^{|x|} \mu_A(x\otimes d_A(y))
 \;\;\mbox{($d_A$ is a derivation)}. $$
%\be{It is denoted by $(A,\mu_A,d_A)$.}
 \end{defi}

 \begin{defi} Given two DG-algebras $(A,\mu_A,d_A)$ and $(B,\mu_B,d_B)$, a DG-algebra morphism 
$f:(A,\mu_A,d_A)\to (B,\mu_B,d_B)$ is a
DG-module morphism satisfying that $$f\mu_{A}=\mu_{B} f.$$
\end{defi}

\begin{defi}
A CDG-algebra is a 
CG- and DG-algebra.
 \end{defi}

  \begin{defi} \cite{19}
 Let $(f,g,\phi)$ a contraction from a DG-algebra
 $(A,\mu_A,d_A)$ to a DG-algebra $(B,\mu_B,d_B)$.
 We say that $\phi$ is an algebra homotopy if
 $$\mu_A (1_A\otimes \phi+\phi\otimes gf)=\phi\mu_A.$$
 \end{defi}
 
 \begin{defi} \cite{pedro}
  We say that  a contraction $(f,g,\phi)$  from a DG-algebra
 $(A,\mu_A,d_A)$ to a DG-algebra $(B,\mu_B,d_B)$ is a full algebra contraction if $f$ and $g$ are DG-algebra morphisms and $\phi$ is an algebra homotopy.
 \end{defi}

Examples of full algebra contractions are given, for example, in \cite{20},  using the ``tensor trick".

\section{Sullivan algebras}

\label{prelim} We recall  basic results and definitions from Rational Homotopy Theory
for which~\cite{GTM2000} is a standard reference.

 A topological space is rational if all their homotopy groups are   $\mathbb Q$-vector spaces.  Rational homotopy can be seen as classical homotopy theory over the rational spaces. 
 Following this theory, each space can be functorially associated with a rational space that satisfies that their homotopy groups are vector spaces on the rational numbers. This process that associates a rational space with each space is called rationalization. This rationalization suppresses the torsion part of the homotopy groups and is, therefore, a first approximation to the given space. From this point of view, the rational homotopy of a space is nothing more than the classical homotopy of its associated rational space. Sullivan showed in \cite{S94,S77} that the rational homotopy type of a simply connected space is faithfully represented by graduated, differential and commutative algebras (CDGAs). In particular, Sullivan defined a functor $F$ that associates each space $X$  with a CDGA $(F(X), d)$ on the rational numbers, with the property of inducing isomorphisms between the respective cohomologies (the one of the algebra $(F(X), d)$ and the one of the space $X$) with rational coefficients. In general, the algebra $(F(X), d)$ is huge and difficult to compute. This is the reason why Sullivan proposed to build a smaller algebra $(M, d)$ from the algebra $(F(X), d)$ called a ``minimal Sullivan algebra" which is a graduated and free algebra over a certain graduated vector space. This minimal model is unique except for isomorphisms and encodes the rational homotopy type of $X$.

From now on,  the ground ring $\Lambda$ is the field of the rational numbers  $\mathbb Q$.
Besides, we will work with DG-algebras  which are free  over graded vector spaces, that is,  for any element $x$ of the algebra, $x\in \Lambda V\oplus \Lambda^{\geq 2} V$ 
for $V$ being a graded vector space $V=\{V^p\}_{p\geq 1}$.
Denote by $\Lambda ^{\geq 2}V$ the G-module generated by elements of $\Lambda V$ obtained as the product of two or more elements of $V$. Let 
$V=\{m_1,\dots,m_n\}$. Then, for any $x\in \Lambda V$, 
$x=\lambda_1 m_1+\cdots+\lambda_n m_n+b$ for some  $b\in \Lambda ^{\geq 2}V$.  The expression coeff$(x,m_i)$ will denote  the coefficient $\lambda_i$.
  Finally,
 we say that $x\in \Lambda V$ has index $j$ and write index$(x)=j$ if $x=m_j$.
We will refer such an algebra by
$(\Lambda V,\mu_{\Lambda V},d_{\Lambda V})$  or simply by 
$\Lambda V$ when no confusion can arise.

\begin{defi}
A DG-algebra $\Lambda V$ is contractible if for some $U \subset  V$ the
inclusions of $\Lambda U\oplus dU$  in $\Lambda V$ extend to an isomorphism:
$$\Lambda (U \oplus dU)  \stackrel{\simeq}{\to} \Lambda V,$$
where $dU=\{d_{\Lambda V}(u): u\in U\}$.
 \end{defi}

\begin{defi}
A Sullivan algebra $(\Lambda V,\mu_{\Lambda V},d_{\Lambda V})$ is a kind of CDG-algbebra:
 \begin{itemize}
\item[(1)] It is a  CG-algebra  which is free  over a graded vector space $V$.
\item[(2)]  There is an increasing filtration of subgraded vector spaces in $V$:
 $$V(0)\subset V(1)\subset \cdots \subset V(\kappa)=V$$
 such that 
 $$\mbox{$d_{\Lambda V}(x)=0$ for any $x\in V(0)$ and}$$ 
 $$\mbox{$d_{\Lambda V}(x)\in \Lambda V(k-1)$ for any $x\in V(k)$, with $1\leq k\leq \kappa$.}$$
\end{itemize}
\end{defi}

 A Sullivan algebra is a free algebra over a vector space with a generator set $V$ and essentially the difference with a minimal Sullivan algebra is that the differential of the minimal algebra has linear terms in the generators of $V$. In other words, being minimal means that the Sullivan algebra has no linear part in the generators of  $V$.

\begin{defi}%
  We say that the Sullivan algebra $(\Lambda V,\mu_{\Lambda V},d_{\Lambda V})$ is {\it minimal} if $d_{\Lambda V}(x) \in \Lambda ^{\geq 2} V$ for all $x\in V$.
  \end{defi}

\begin{te}\cite{GTM2000}
For every rational homotopy type of spaces there is a unique (up to isomorphism) minimal Sullivan algebra. 
\end{te}
 
  An essential property of  Sullivan algebras is that they are isomorphic to a minimal algebra tensor product with a contractile algebra and, therefore, removing the contractile part, we obtain its minimal model.

 \begin{te}[Theorem 14.9 of \cite{felx}]
Every Sullivan algebra $\Lambda V$  is isomorphic to $\Lambda W\otimes \Lambda (U \oplus dU)$
 where $\Lambda W$ is a minimal Sullivan algebra and $\Lambda (U\oplus dU)$ is contractible.
\end{te}

\section{An algorithm to compute minimal Sullivam algebras}

 Essentially what is done in the literature for calculating a minimal Sullivan algebra (see, for example, \cite{GTM2000}) is a change of basis in a way that   the algebra is divided in a contractible part and a non-contractible part, and the differential of the non-contractible part has no linear terms. Therefore, the non-contractible part is a minimal Sullivan algebra having the same cohomology of the given algebra indicating that it is the minimal model of the given algebra.

Our main goal in this section  is to provide an efficient algorithm for computing a minimal Sullivan algebra from a Sullivan algebra not necessarily being minimal.
It is precisely the fact that the Sullivan algebra and its minimal model have the same cohomology (since they are only different in a contractible part) what made us relate this problem to another apparently different in principle: The incremental algorithm for computing an "AT-model" of a chain complex (see Algorithm \ref{alg}).

\begin{algorithm2e}[H]
{\bf Input: } A Sullivan algebra $\Lambda V$,
being 
$V(0)\subset V(1)\subset\cdots \subset V(\kappa) =V$ a filtration and 
$V=V(0)\cup \{m_1, \ldots, m_n\}$ 
such that if $1\leq i<i'\leq n$ then $m_i\in V(k)$ and $m_{i'}\in V(k')$ 
where $1\leq k\leq k'\leq \kappa$.

\BlankLine

$W:=V(0)$, 
$f:=id_{V(0)}$, $g:=id_{V(0)}$, $\phi:=0$ and $d_{\Lambda W}:=0$
\\
\For{$i=1$ {\bf to} $n$}
{
let $a:= f d_{\Lambda V}  (m_i)$ and 
	$b:= m_i-\phi d_{\Lambda V}  (m_i)$ \\
{\If{$a\in \Lambda ^{\geq 2} W$}
    {(the element $m_i$ is added to $W$)\\
    $W:=W\cup\{m_i\}$, $f(m_i):=m_i, \; g(m_i):=b
    $,
  $\phi(m_i):=0$ and $d_{\Lambda W}(m_i):=a
  $}
{\If {
$a \notin \Lambda ^{\geq 2} H$
}
	{
	let 
$j=\max\{$index$(m): m\in W$ and 
coeff$(a,m)\neq 0\}$\\
(the element $m_j$ is removed from $W$ and $(m_i,m_j)$ are paired) \\
$W:=W\setminus\{m_j\}$, $f(m_i):=0$ and $\phi(m_i):=0$ \\
	\ForEach{$m\in V(0) \cup \{m_1,\dots,m_{i-1}\}$ %s.t. $m_j\in f(x)$
	}
	    {$f(m):=f(m)-\lambda a$
	    and
	    $\phi(m):=\phi(m)+\lambda b$ where  
	    $\lambda:=\frac{\mbox{\tiny coeff}(f(m),m_j)}{\mbox{\tiny coeff}(a,m_j)}$
	    }
    } 
    }
     $f \mu_{\Lambda V}:=\mu_{\Lambda W} (f\otimes f)$,
$g \mu_{\Lambda W}:= \mu_{\Lambda V} (g\otimes g)$,
$\phi \mu_{\Lambda V}:=\mu_{\Lambda V}(id_{\Lambda V} \otimes \phi + \phi \otimes g f)$ and $d_{\Lambda W} \mu_{\Lambda W}:= f d_{\Lambda V} \mu_{\Lambda V}(g\otimes g)$
}
}
\BlankLine
{\bf Output: } A full algebra 
contraction $(f,g,\phi)$ from $\Lambda V$ to $\Lambda W$.
\caption{}\label{alg:minimal}
\end{algorithm2e}

\begin{te}\label{th:main}
Given a Sullivan algebra $(\Lambda V,\mu_{\Lambda V},d_{\Lambda V})$, Algorithm \ref{alg:minimal}  produces a
a full algebra contraction
$(f,g,\phi)$  from 
$\Lambda V$ to $\Lambda W$. Moreover, $\Lambda W$ is a minimal Sullivan algebra.
\end{te}

\begin{proof}
Denote   the output of Algorithm \ref{alg:minimal} at the step $i$
by $W_i$, $f_i$, $g_i$ and $\phi_i$. 
\\
First, it is straightforward to see that $(f_{0},g_{0},\phi_{0})$ is a full algebra contraction from $\Lambda V(0)$ to $\Lambda W_{0}$. 
\\
Now, by induction, suppose that  $(f_{i-1},g_{i-1},\phi_{i-1})$ is a full algebra contraction from $\Lambda V_{i-1}$ to $\Lambda W_{i-1}$ where 
$V_{i-1}:=V(0)\cup\{m_1,\dots,m_{i-1}\}$. 
Let us prove that $(f_{i},g_{i},\phi_{i})$ is a full contraction from $\Lambda V_i$ to $\Lambda W_{i}$ where 
$V_{i}:=V_{i-1}\cup\{m_{i}\}$. To start with, let us prove the following properties by induction:
\begin{enumerate}
    \item $f_i\phi_i=0$
    \item $\phi_i g_i=0$
    \item $\phi_i \phi_i=0$
    \item $f_i g_i=id_{\Lambda W_i}$
    \item $id_{\Lambda V_i} - g_if_i = \phi_i d_{\Lambda V_i} + d_{\Lambda V_i} \phi_i$
    \item $f_i d_{\Lambda V_i}=d_{\Lambda W_i}f_i$
    \item $d_{\Lambda V_i} g_i=g_i d_{\Lambda W_i}$
\end{enumerate}
Suppose first that $f_{i-1}d_{\Lambda V_i} (m_i) \in \Lambda ^{\geq 2} W_{i-1}$. In this case, it is enough to prove the properties above only for $m_i$: 

\begin{enumerate}
\item $f_i\phi_i(m_i)=0$ since $\phi_i(m_i)=0$.
\item $\phi_i g_i(m_i)=\phi_i(m_i-\phi_{i-1}d_{\Lambda V_i}(m_i))=\phi_i(m_i)-\phi_{i-1}\phi_{i-1}d_{\Lambda V_i}(m_i)=0$ since $\phi_i(m_i)=0$ and $\phi_{i-1}\phi_{i-1}=0$ by induction.
%For $m_i$, $\phi_i %g_i(m_i)=\phi_i(m_i-\phi_{i-1}d_{\Lambda %V}(m_i))=\phi_i(m_i)-\phi_{i-1}\phi_{i-1}d_{\Lambda %V}(m_i)=0$ since $\phi_i(m_i):=0$ and by induction.\\
%For $m \in H_{i-1}$, $\phi_i %g_i(m)=\phi_{i-1}g_{i-1}(m)=0$ by induction.
%
\item $\phi_i\phi_i(m_i)=0$ since $\phi_i(m_i)=0$.
\item $f_i g_i(m_i) = f_i(m_i-\phi_{i-1}d_{\Lambda V_i}(m_i))=f_i(m_i)-f_{i-1}\phi_{i-1}d_{\Lambda V_i}(m_i)=f_i(m_i)=m_i$
since $f_{i-1}\phi_{i-1}=0$ by induction.
\item 
$\phi_{i}d_{\Lambda V_i}(m_i) + d_{\Lambda V_i}\phi_{i}(m_i)=
\phi_{i-1}d_{\Lambda V_i}(m_i)=m_i - g_i(m_i)=m_i - g_{i}f_{i}(m_i)$.
%$m_i - g_{i}f_{i}(m_i)=m_i - g_i(m_i)=m_i-(m_i-\phi_{i-1}d_{\Lambda V_i}(m_i))= \phi_{i-1}d_{\Lambda V_i}(m_i) = \phi_{i}d_{\Lambda V_i}(m_i) =$ since $\phi_i(m_i):=0$.
%For $m_i$, $m_i - g_{i}f_{i}(m_i)=m_i - %g_i(m_i)=m_i-(m_i-\phi_{i-1}d_{\Lambda V}(m_i))= %\phi_{i-1}d_{\Lambda V}(m_i) = \phi_{i}d_{\Lambda %V}(m_i) =\phi_{i}d_{\Lambda V}(m_i) + d_{\Lambda %V}\phi_{i}(m_i)$ since $\phi_i(m_i):=0$.\\
%For $m \in V(0) \cup \{m_1, \dots,m_{i-1} \}$, $m - %g_{i}f_{i}(m)=m - %g_{i-1}f_{i-1}(m)=\phi_{i-1}d_{\Lambda V}(m) + %d_{\Lambda V}\phi_{i-1}(m) = \phi_{i}d_{\Lambda V}(m) + %d_{\Lambda V}\phi_{i}(m)$ by induction.
%
\item $d_{\Lambda W_{i}} f_i(m_i)=d_{\Lambda W_{i}} (m_i)=f_{i-1} d_{\Lambda V_i}(m_i)=f_{i} d_{\Lambda V_i}(m_i)$.
\item $g_i d_{\Lambda W_{i}}(m_i)=
%g_i f_i d_{\Lambda V_i} g_i(m_i)= 
g_i f_{i-1} d_{\Lambda V_i}(m_i)=
g_{i} f_{i} d_{\Lambda V_i}(m_i)=
d_{\Lambda V_i}(m_i) - \phi_{i}d_{\Lambda V_{i}} d_{\Lambda V_i}(m_i) - d_{\Lambda V_{i}} \phi_{i} d_{\Lambda V_i}(m_i)=
d_{\Lambda V_i}(m_i) -  d_{\Lambda V_{i}} \phi_{i} d_{\Lambda V_i}(m_i)=
d_{\Lambda V_i}(m_i - \phi_{i-1} d_{\Lambda V_i}(m_i))=
d_{\Lambda V_i} g_i(m_i)$.
\end{enumerate} 
Second, suppose that 
%$f_{i-1} d_{\Lambda V_i} (m_i)\neq 0$ and 
$f_{i-1} 
d_{\Lambda V_i} (m_i) \notin \Lambda ^{\geq 2} W_{i-1}$. Let $j=\max\{$index$(m): 
m\in W_{i-1}$ and coeff$(f_{i-1} d_{\Lambda V_i}(m_i),m)\neq 0\}$.
%$m\in f_{i-1} d_{\Lambda V_i}(m_i)$ and $m$ is simple$\}$. }
%
Let $m\in V_{i-1}$ then:
%$f_i(m):=f_{i-1}(m)$ and $\phi_i(m):=\phi_{i-1}(m)$ for $m \in V(0) \cup \{m_1, \dots,m_{i-1} \}$ s.t. $m_j \notin f_{i-1}(m)$ and $g_i(m):=g_{i-1}(m)$ for $m \in H_i$. Let $m_k \in V(0) \cup \{m_1, \dots,m_{i-1} \}$ s.t. $m_j \in f_{i-1}(m_k)$ and let $b:=$coeff$(f_{i-1}(m_k),m_j)$. Then,
%
\begin{enumerate}
\item 
$f_i \phi_i(m_i)=0$ since $\phi_i(m_i)=0$;\\
$f_i\phi_i(m)= f_i(\phi_{i-1}(m)+\lambda(m_i-\phi_{i-1}d_{\Lambda V_i}(m_i)) = f_{i-1}\phi_{i-1}(m)+\lambda f_i(m_i)-\lambda f_{i-1}\phi_{i-1}d_{\Lambda V_i}(m_i)=0$ since $f_i(m_i)=0$ and $f_{i-1}\phi_{i-1}=0$ by induction.
\item $\phi_i g_i(m)=\phi_{i-1} g_{i-1}(m)=0$
since  $\phi_{i-1}g_{i-1}=0$ by induction.
\item $\phi_i \phi_i(m_i)=0$ since $\phi_i(m_i)=0$;\\
$\phi_i \phi_i(m)=\phi_i(\phi_{i-1}(m)+\lambda(m_i-\phi_{i-1}d_{\Lambda V_i}(m_i)) = \phi_{i-1}\phi_{i-1}(m)+\lambda\phi_i(m_i)-\lambda\phi_{i-1}\phi_{i-1}d_{\Lambda V_i}(m_i)=0$ since $\phi_i(m_i)=0$ and $\phi_{i-1}\phi_{i-1}=0$ by induction.
\item $f_i g_i(m)= f_{i-1} g_{i-1}(m)= m$
since $f_{i-1} g_{i-1}=id_{\Lambda W_{i-1}}$ by induction.
\item 
$\phi_i d_{\Lambda V_i}(m_i) + d_{\Lambda V_i} \phi_i(m_i) = \phi_{i-1}d_{\Lambda V_i}(m_i) + m_i - \phi_{i-1}d_{\Lambda V_i}(m_i) =  m_i - g_if_i(m_i)$;
\\
$\phi_i d_{\Lambda V_i}(m) + d_{\Lambda V_i} \phi_i(m) =\phi_{i-1} d_{\Lambda V_i}(m) + d_{\Lambda V_i}(\phi_{i-1}(m) + \lambda (m_i -  \phi_{i-1}d_{\Lambda V_i}(m_i))
%= \phi_i d_{\Lambda V_i}(m) + d_{\Lambda V_i}\phi_{i-1}(m) + \lambda d_{\Lambda V_i}(m_i) - \lambda d_{\Lambda V_i}\phi_{i-1}d_{\Lambda V_i}(m_i)
=m - g_{i-1}f_{i-1}(m)+\lambda g_{i-1}f_{i-1}d_{\Lambda V_i}(m_i) 
%= m - g_{i-1}(f_{i-1}(m)-\lambda f_{i-1}d_{\Lambda V_i}(m_i))
= m - g_if_i(m)$.
\item $f_i d_{\Lambda V_i}(m_i)=f_{i-1}d_{\Lambda V_i}(m_i) - f_{i-1}d_{\Lambda V_i}(m_i)=0=d_{\Lambda W_i} f_i(m_i)$;
\\
$d_{\Lambda W_i} f_i(m)=d_{\Lambda W_i}(f_{i-1}(m)-\lambda f_{i-1}d_{\Lambda V_i}(m_i))
%=d_{\Lambda H_{i-1}}f_{i-1}(m)-\lambda d_{\Lambda W_{i-1}}f_{i-1}d_{\Lambda V_i}(m_i)
%=f_{i-1}d_{\Lambda V_i}(m) - \lambda f_{i-1}d_{\Lambda V_i}d_{\Lambda V_i}(m_i)
=f_{i-1}d_{\Lambda V_i}(m)=f_{i}d_{\Lambda V_i}(m)$.
\item $g_i d_{\Lambda W_i}(m)=g_{i-1} d_{\Lambda 
W_{i-1}}(m) =
 d_{\Lambda V_{i-1}}g_{i-1} (m)=
d_{\Lambda V_i} g_i(m)$.
\end{enumerate} 
Now, it is easy to see that $(\Lambda W_i,d_{\Lambda W_i})$ is a DG-module, that is, $d_{\Lambda W_i}$ is a differential. It is enough to prove it for $m_i$: $d_{\Lambda W_i}d_{\Lambda W_i}(m_i) = d_{\Lambda W_i} f_{i-1}d_{\Lambda V_i} (m_i) = d_{\Lambda W_i} f_{i}d_{\Lambda V_i} (m_i) = f_i d_{\Lambda V_i} d_{\Lambda V_i} (m_i) = 0.$
\\
Therefore, we conclude that the output $(f,g,\phi)$ of Algorithm \ref{alg:minimal} is a contraction from $\Lambda V$ to $\Lambda W$. 
\\
To prove that $\Lambda W$ is an algebra,
we have to prove the following properties:
\begin{enumerate}
    \item $d_{\Lambda W}$ is derivation.
    \item $1_{\Lambda W}:=f(1_{\Lambda V})$ is an identity.
    \end{enumerate}
      Using that $(f,g,\phi)$ is a contraction and that,  by construction, we have that 
      $d_{\Lambda V} \mu_{\Lambda V} = fd_{\Lambda V}\mu_{\Lambda V} (g\otimes g)$, $f \mu_{\Lambda V} = \mu_{\Lambda W} (f\otimes f)$, $g \mu_{\Lambda W} = \mu_{\Lambda V} (g\otimes g)$ and $\phi \mu_{\Lambda V} = \mu_{\Lambda V}(id_{\Lambda V} \otimes \phi + \phi \otimes g f)$, then:
\begin{enumerate}
   % \item     $\mu_{\Lambda H_i}(x\otimes \mu_{\Lambda H_i}(y\otimes z)) =
    %f_i \mu_{\Lambda V_i}(g_i \otimes g_i)(x\otimes f_i \mu_{\Lambda V_i}(g_i \otimes g_i)(y\otimes z))
    %=f_i \mu_{\Lambda V_i}(g_i(x)\otimes g_if_i \mu_{\Lambda V_i}(g_i(y) \otimes g_i(z))=f_i \mu_{\Lambda V_i}(g_i(x)\otimes \mu_{\Lambda V_i}(g_i(y) \otimes g_i(z))
    %=f_i \mu_{\Lambda V_i}( \mu_{\Lambda V_i}(g_i(x) \otimes g_i(y))\otimes g_i(z))
    %=\mu_{\Lambda W_i}(\mu_{\Lambda W_i}(x\otimes y)\otimes z)     $.
    %
    %$f_i \mu_{\Lambda V_i}(g_i \otimes g_i)(x\otimes \mu_{\Lambda H_i}(y\otimes z))=f \mu_{\Lambda V_i}(g_i(x)\otimes g_i \mu_{\Lambda H_i}(y\otimes z))=f_i \mu_{\Lambda V_i}(g_i(x)\otimes \mu_{\Lambda V_i} g_i (y\otimes z))=f_i \mu_{\Lambda V_i}(g_i(x)\otimes \mu_{\Lambda V_i} (g_i (y)\otimes g_i(z)))=f_i \mu_{\Lambda V_i} (\mu_{\Lambda V_i}(g_i(x)\otimes g_i(y))\otimes g_i(z))=f \mu_{\Lambda V_i} (g_i \mu_{\Lambda H_i} (x \otimes y) \otimes g_i(z)) = f_i \mu_{\Lambda V_i} (g_i \otimes g_i) (\mu_{\Lambda H_i}(x \otimes y) \otimes z)=\mu_{\Lambda H_i}(\mu_{\Lambda H_i}(x \otimes y) \otimes z).$
    %\be{Usa $g \mu_{\Lambda W}=\mu_{\Lambda V} g$ ($g$ morfismo de %DG-\'algebras) y la propiedad asociativa para $\mu_{\Lambda V}$.}
    %%
        %%
    \item 
    $d_{\Lambda W}\mu_{\Lambda W} (x\otimes y)= f d_{\Lambda V} \mu_{\Lambda V} (g(x)\otimes g(y)) =\\ f \mu_{\Lambda V}(d_{\Lambda V} g(x) \otimes g(y)) + (-1)^{|g(x)|} f \mu_{\Lambda V}(g(x) \otimes d_{\Lambda V} g(y)) =\\ 
    \mu_{\Lambda W} (f d_{\Lambda V}g(x) \otimes fg(y)) +
    (-1)^{|g(x)|} \mu_{\Lambda W}
    (f g(x) \otimes f d_{\Lambda V} g(y)) =\\ 
    \mu_{\Lambda W} (d_{\Lambda W} fg(x) \otimes fg(y)) +     (-1)^{|x|} \mu_{\Lambda W}
    (f g(x) \otimes d_{\Lambda W} fg(y)) =\\ 
    \mu_{\Lambda W} (d_{\Lambda W} (x) \otimes y) + (-1)^{|x|} \mu_{\Lambda W}
    (x \otimes d_{\Lambda W} (y))$ since $d_{\Lambda V}$ is a derivation and $|g|=0$.
      % f g d_{\Lambda W} f \mu_{\Lambda V} (g \otimes g) (x \otimes y) = f d_{\Lambda V} \mu_{\Lambda V} (g \otimes g) (x \otimes y) = f d_{\Lambda V} \mu_{\Lambda V} (g(x) \otimes g(y))= f \mu_{\Lambda V}(d_{\Lambda V}g(x) \otimes g(y)) + (-1)^{|g(x)|} f \mu_{\Lambda V}(g(x) \otimes d_{\Lambda V} g(y)) =  \mu_{\Lambda W} f (g d_{\Lambda W}(x) \otimes g(y)) + (-1)^{|x|}  \mu_{\Lambda W} f (g(x) \otimes g d_{\Lambda W} (y)) = \mu_{\Lambda W} (d_{\Lambda W}(x) \otimes y) + (-1)^{|x|}  \mu_{\Lambda W} (x \otimes d_{\Lambda W} (y)).$\\
    %\be{Se usa $fg=1_{\Lambda W}$, $f d_{\Lambda V} = d_{\Lambda W} f$, %$g d_{\Lambda W} = d_{\Lambda V} g$ ($f$ and $g$ morfismos de %DG-m\'odulo) y $f \mu_{\Lambda V} = \mu_{\Lambda W} f$ ($f$ morfismo %de DG-\'algebra).}
      % where $\lambda:=\frac{\mbox{coeff}(f_{i-1}(m),m_j)}{\mbox{coeff}(f_{i-1} d_{\Lambda V_i}(m_i),m_j)}$
       % \item 
    %$\mu_{\Lambda W}(a_p\otimes a_q)=f \mu_{\Lambda V}(g(a_p)\otimes g(a_q)) = (-1)^{pq} f \mu_{\Lambda V}(g(a_q)\otimes g(a_p)) = (-1)^{pq} f \mu_{\Lambda V}(g \otimes g)(a_q\otimes a_p) = (-1)^{pq} \mu_{\Lambda W}(a_q\otimes a_p)$, where $a_p\in W^p$ and $a_q\in W^q$.
    %\be{Usa la propiedad para $\mu_{\Lambda V}$ y que $g$ es de grado %0.}
    %%
   \item 
    $\mu_{\Lambda W}(1_{\Lambda W}\otimes x)= \mu_{\Lambda W}(f(1_{\Lambda V})\otimes f g(x)) = f \mu_{\Lambda V} (1_{\Lambda V} \otimes g(x))=f g(x))=x$ since $1_{\Lambda V}$ is an identity and $fg=1_{\Lambda W}$.
    %$\mu_{\Lambda W_i}(x\otimes 1_{\Lambda W_i})=f_i \mu_{\Lambda V_i} (g_i \otimes g_i)(x\otimes 1_{\Lambda W_i})=f_i \mu_{\Lambda V_i}(g_i(x)\otimes g_i(1_{\Lambda W_i}))=f_i \mu_{\Lambda V_i}(g_i(x)\otimes 1_{\Lambda V_i})=f_ig_i(x)=x$.
    %\be{Usa $fg=id_{\Lambda W}$, $g(1_{\Lambda W})=1_{\Lambda V}$ y la %propiedad para $\mu_{\Lambda V}$.}
    %%
    \end{enumerate}
 Therefore, we can conclude that  $\Lambda W$ is a CDG-algebra and $(f,g,\phi)$ is a full algebra contraction from $\Lambda V$ to $\Lambda W$.
 \\
 Finally, $\Lambda W$ is a Sullivan algebra  considering the filtration of $V$ restricted to $W$.  Besides, $\Lambda W$ is minimal by induction.
 Trivially, $W_0$ is minimal since $d_{\Lambda W_0}=0$. Suppose that
$\Lambda W_{i-1}$ is minimal. Now, 
$d_{\Lambda W_{i}}$ is updated when 
$f_{i-1}d_{\Lambda V} (m_i) \in \Lambda ^{\geq 2} W_{i-1}$ and in this case, 
%it is enough to prove that   $d_{\Lambda W_i}g(m_i)\in \Lambda ^{\geq 2} W_{i}$:
$d_{\Lambda W_{i}}(m_i) =  f_{i-1} d_{\Lambda V_i} (m_i)\in  \Lambda ^{\geq 2}W_{i}$.
\end{proof}

Observe that   
$g$ is one-to-one due to $fg=id_{\Lambda W}$ then, we can obtain a base of generators $g(W)\oplus X$ 
for   $\Lambda V$, with the property that 
$\Lambda X$ is contractible (since $\Lambda V$ and $\Lambda W$ have the same  rational cohomology, due to the contraction from $\Lambda V$ to $\Lambda W$).
Besides, by construction, $\Lambda X$ is isomorphic to $\lambda (U\oplus dU)$ where $U\oplus dU$ is composed by the set of all the pairs $\{(m_i,m_j)\}$ created when running Algorithm \ref{alg:minimal} on $V$.

\section{Some examples}
Below we show some examples of the computation of minimal  Sullivan algebras using Algorithm \ref{alg:minimal}. A naive implementation of the algorithm can be consulted in http://grupo.us.es/cimagroup/downloads.htm.

For the sake of simplicity, from now on, given an algebra $\Lambda V$ and two elements $a,b\in V$, the element $\mu_{\Lambda V}(a\otimes b)$ will be denoted by $ab$ when no confusion can arise.

\begin{ex}\label{ex1}
Consider the Sullivan algebra $\Lambda V$ where
$V =\{V^p\}_{p\geq 1}$ being:  
$$V^1= \{a_1, b_1, c_1\},\; V^2=\{v_2\}, \;V^3=\{u_3\},$$
 differential $d_{\Lambda V}$  defined on $V$ as follows: 
$$d_{\Lambda V}(a_1)=v_2,\;
d_{\Lambda V}(b_1)=0=d_{\Lambda V}(c_1),\;d_{\Lambda V}(u_3)=v_2^2,$$ 
and filtration:  $$V(0)=\{b_1, c_1, v_2\}\subset V(1)=V.$$ 
Then, Algorithm \ref{alg:minimal} runs as follows.
\\
Initially, $W=V(0)$, $f=id_{\Lambda V(0)}$, $g=id_{\Lambda V(0)}$, $\phi=0$ and $d_{\Lambda W} = 0$. 
    %$$\begin{tabular}{c|c|c|c|c}
    %     $V(0)$ & $H$ & $f$ & $g$ & $\phi$\\
     %    \hline
     %    $b_1$ & $b_1$ & $b_1$ & $b_1$ & $0$\\ 
     %    $c_1$ & $c_1$ & $c_1$ & $c_1$ & $0$\\ 
     %    $v_2$ & $v_2$ & $v_2$ & $v_2$ & $0$\\ 
    %\end{tabular}$$
    %
    \item Now, $f d_{\Lambda V}(a_1)=f(v_2)=v_2  \notin \Lambda ^{\geq 2} W$. Then
    $W$ and the image of the morphisms $f$, $g$ and $\phi$ for each generator are updated as follows:
    %\begin{itemize}
    %\item $W=\{b_1, c_1\}$.
    %\item $f(v_2)=f(v_2)-f d_{\Lambda V}(a_1) = 0$.
    %\item $\phi(v_2)=\phi(v_2) + (a_1 - \phi %d_{\Lambda V}(a_1))=a_1$.
    %\end{itemize}
    $$\begin{tabular}{c|c|c|c|c|c|c|}
         $V$ &  $d_{\Lambda V}$ & $W$ & $d_{\Lambda W}$ & $f$ & $g$ & $\phi$\\
         \hline
         $b_1$ & $0$ & $b_1$ & $0$ & $b_1$ & $b_1$ & $0$\\
         $c_1$ & $0$ & $c_1$ & $0$ & $c_1$ & $c_1$ & $0$\\ 
         $v_2$ & $0$ & & &$0$ &  & $a_1$\\
         $a_1$ & $v_2$ & && $0$ &  & $0$\\
         $u_3$ &$v_2^2$ & & & $0$ &  & $0$
    \end{tabular}$$
    Finally, $f d_{\Lambda V}(u_3)=f(v_2^2)
    %=f \mu_{\Lambda V}(v_2 \otimes v_2)
    = f(v_2) f(v_2) = 0$ since $f(v_2)=0$. Then $W$ and the image of the morphisms $f$, $g$ and $\phi$ for each generator are updated as follows:
    %\begin{itemize}
    %\item[] $W=\{b_1, c_1, u_3\}$.
    %\item[] $f(u_3)=u_3$.
    %\item[] $g(u_3) = u_3 - \phi d_{\Lambda %V}(u_3)=u_3-\phi(v_2^2)
    %%=u_3-\phi \mu_{\Lambda V}(v_2 \otimes v_2) 
    %= u_3 -\mu_{\Lambda V}(v_2 \otimes \phi(v_2) + %\phi(v_2) \otimes g f(v_2)) =
    %u_3-a_1v_2$.
   %% u_3 - \mu_{\Lambda V}(v_2 \otimes a_1)=u_3 - %\mu_{\Lambda V}(a_1 \otimes v_2)$.
    %\item[] $\phi(u_3)=0$.
    %\item[] $d_{\Lambda W}(u_3)=f d_{\Lambda V} g %(u_3)=f d_{\Lambda V}(u_3 - \mu_{\Lambda V}(a_1 %\otimes v_2))=f d_{\Lambda V}(u_3) - %f\mu_{\Lambda V}(a_1 \otimes v_2)=0$.
        %\end{itemize}
    %$$\begin{tabular}{c|c|c|c|c}
    %     $V(0)$ & $H$ & $f$ & $g$ & $\phi$\\
    %     \hline
    %     $b_1$ & $b_1$ & $b_1$ & $b_1$ & $0$\\ 
    %     $c_1$ & $c_1$ & $c_1$ & $c_1$ & $0$\\ 
    %     $v_2$ &  & $0$ &  & $a_1$\\
    %     $a_1$ &  & $0$ &  & $0$\\
    %     $u_3$ & $u_3$ & $u_3$ & $u_3-a_1v_2$ & $0$\\
    %\end{tabular}$$
        $$\begin{tabular}{c|c|c|c|c|c|c|}
         $V$ &  $d_{\Lambda V}$ & $W$ & $d_{\Lambda W}$ & $f$ & $g$ & $\phi$\\
         \hline
         $b_1$ & $0$ & $b_1$ & $0$ & $b_1$ & $b_1$ & $0$\\
         $c_1$ & $0$ & $c_1$ & $0$ & $c_1$ & $c_1$ & $0$\\ 
         $v_2$ & $0$ & & &$0$ &  & $a_1$\\
         $a_1$ & $v_2$ & && $0$ &  & $0$\\
         $u_3$ &$v_2^2$ & $u_3$ & $0$ & $u_3$ & $u_3-a_1 v_2$ & $0$
    \end{tabular}$$
    Therefore, $\Lambda W$ with set  of generators $W=\{b_1,c_1,u_3\}$ 
    is a minimal Sullivan algebra. Besides,  
   take the pair $\{(a_1,v_2)\}$, obtained when running Algorithm \ref{alg:minimal} on $V$ and 
    write $U=\{a_1\}$. We then finally obtain that $\Lambda V$ is isomorphic to $\Lambda W\otimes\Lambda (U\oplus dU) $ being  $\Lambda W$ a minimal Sullivan algebra and $\Lambda (U\oplus dU)$ contractible. Besides, the basis provided by the ``ad hoc" method can be obtained using morphism $g$. This way, the new basis of generators of $\Lambda V$ is:
    $$\{g(b_1)=b_1,\;  g(c_1)=c_1,\;  g(u_3)=u_3-a_1v_2\}$$
    with differential:
    $$d_{\Lambda V}g(b_1)=0,\;  d_{\Lambda V}g(c_1)=0,\;  d_{\Lambda V}g(u_3)=0.$$
        \end{ex}

\begin{ex}
If we consider the Sullivan algebra $\Lambda V$ where $V$ is the same as the one given in Example \ref{ex1}, differential $d_{\Lambda V}$ defined on $V$ as follows:
$$d_{\Lambda V}(a_1)=d_{\Lambda V}(b_1)=d_{\Lambda V}(c_1)=v_2,\; d_{\Lambda V}(v_2)=0,\; d_{\Lambda V}(u_3)=v_2^2$$
and filtration 
$$V(0)=\{v_2\}\subset V(1)=V,$$
then Algorithm \ref{alg:minimal} produces 
the same minimal Sullivan algebra and the same full algebra contraction as in Example \ref{ex1}.
\end{ex}

\begin{ex}
Consider the Sullivan algebra $\Lambda V$ where $V = \{V^p\}_{p\geq 1}$ being:
$$V^1=\{a_1, b_1, c_1,x_1,y_1\}, \; V^2=\{v_2, p_2, q_2, r_2\}, \;V^3=\{u_3\},$$
differential $d_{\Lambda V}$  defined on $V$ as follows: 
$$d_{\Lambda V}(a_1)=d_{\Lambda V}(b_1)=d_{\Lambda V}(c_1)=d_{\Lambda V}(v_2)=0,\;$$
$$d_{\Lambda V}(x_1)=v_2-2a_1b_1+2b_1c_1,\;d_{\Lambda V}(y_1)=v_2-2a_1c_1-2b_1c_1,\;d_{\Lambda V}(p_2)=2v_2a_1,$$
$$d_{\Lambda V}(q_2)=2v_2b_1,\;d_{\Lambda V}(r_2)=2v_2c_1,\;d_{\Lambda V}(u_3)=v_2^2,$$
and filtration: 
$$V(0)=\{a_1,b_1,c_1,v_2\}\subset V(1)=V.$$
Then, the output of Algorithm \ref{alg:minimal} is:
$$\begin{tabular}{c|c|c|c|c|c|}
         $V$ 
         %& $d_{\Lambda V}$ 
         & $ W$ &$d_{\Lambda W}$&$f$ & $g$ & $\phi$\\
         \hline
         $a_1$ 
        % &$0$
         &         $a_1$ &$ 0$& 
         $a_1$ & $a_1$ & $0$\\ 
         $b_1$ 
         %&  $0$
         & $b_1$ &$0 $&  
         $b_1$ & $b_1$ & $0$\\ 
         $c_1$ 
         %&  $0$
         & $c_1$ &$0 $& 
         $c_1$ & $c_1$ & $0$\\
         $v_2$ 
         %&  $0$ 
         &&$ $& 
         $2a_1b_1-2b_1c_1$ &  & $x_1$\\
         $x_1$ 
         %&  $v_2-2a_1b_1+2b_1c_1$ 
         & &$ $& 
         $0$ &  & $0$\\
         $y_1$ 
         %&  $v_2-2a_1c_1-2b_1c_1$
         & $y_1$ &$2a_1b_1-2a_1c_1-4b_1c_1 $& 
         $y_1$ & $y_1-x_1$ & $0$\\ 
         $p_2$ 
         %&  $2v_2a_1$ 
         & $ p_2$ &$-4a_1b_1c_1 $& 
         $p_2$ & $p_2-2a_1x_1$ & $0$\\ 
         $q_2$ 
         %&  $2v_2b_1$ 
         & $q_2$ &$ 0$& 
         $q_2$ & $q_2-2b_1x_1$ & $0$\\
         $r_2$ %&  $2v_2c_1$ 
         & $r_2$ &$ 4a_1b_1c_1 $& 
         $r_2$ & $r_2-2c_1x_1$ & $0$\\
         $u_3$ %&  $v_2^2$
         & $u_3$ &$0 $& 
         $u_3$ & $u_3-x_1v_2$ & $0$\\
         &&&&$-2a_1b_1x_1+2b_1c_1x_1$&
    \end{tabular}$$
    Therefore, $\Lambda W$ with $W=\{a_1,b_1,c_1,y_1,p_1,q_1,r_1u_3\}$ is a minimal Sullivan algebra with the same (co)homology  than $\Lambda V$. Besides, take the pair $\{(x_1,v_2)\}$ obtained when running Algorithm \ref{alg:minimal}. 
    If  we denote $v'_2:=v_2-2a_1b_1 $ and write $U=\{x_1\}$, then $dU=\{v'_2\}$ and we finally obtain that $\Lambda V$ is isomorphic to $\Lambda W\otimes\Lambda (U\oplus dU) $ being  $\Lambda W$ a minimal Sullivan algebra and $\Lambda (U\oplus dU)$ contractible.
    Besides, the basis provided by the ``ad hoc" method can be obtained using morphism $g$:
    $$\big\{g(a_1)=a_1,\; g(b_1)=b_1,\;  g(c_1)=c_1,\;  g(y_1)=y_1-x_1, \; 
    g(p_2)=p_2-2a_1x_1,$$
   $$ g(q_2)=q_2-2b_1x_1,\;  g(r_2)=r_2-2c_1x_1,\;  g(u_3)=u_3-x_1v_2-2a_1b_1x_1+2b_1c_1x_1\big\}$$
    with differential:
   $$d_{\Lambda V}g(a_1)=d_{\Lambda V}g(b_1)= d_{\Lambda V}g(c_1)=
   d_{\Lambda V}g(q_2)=d_{\Lambda V}g(u_3)=0,$$
  $$d_{\Lambda V}g(y_1)=2g(a_1)g(b_1)-2g(a_1)g(c_1)-4g(b_1)g(c_1),$$
    $$d_{\Lambda V}g(p_2)=-4g(a_1)g(b_1)g(c_1),\;
  d_{\Lambda V}g(r_2)=4g(a_1)g(b_1)g(c_1). $$   
    \end{ex}

\begin{ex}
Consider the Sullivan algebra $\Lambda V$ where $V = \{V^p\}_{p\geq 1}$ being:
$$V^1=\{x_1\}, \; V^2=\{v_2,w_2\}, \;V^3=\{x_3\},\; V^4=\{v_4,w_4\},\; V^5=\{x_5\}, \; V^7=\{x_7\},$$
and differential $d_{\Lambda V}$  defined on $V$ as follows: 
$$d_{\Lambda V}(x_1)=v_2 + w_2,\;d_{\Lambda V}(x_3)=v_4 + w_4 + v_2 w_2,\;d_{\Lambda V}(x_5)=v_4 w_2 + v_2 w_4$$
$$d_{\Lambda V}(x_7)=v_4 w_4,\; d_{\Lambda V}(v_2)=d_{\Lambda V}(w_2)=d_{\Lambda V}(v_4)=d_{\Lambda V}(w_4)=0,$$
and filtration: 
$$V(0)=\{v_2, w_2, v_4, w_4\}\subset V(1)=V.$$
Then, the output of Algorithm \ref{alg:minimal} is:
$$\begin{tabular}{c|c|c|c|c|c|}
         $V$ & $W$ & $d_{\Lambda W}$ & $f$ & $g$ & $\phi$\\
         \hline
         $v_2$ &  $v_2$ & $0$ & $v_2$ & $v_2$ & $0$\\ 
         $w_2$ &  &  & $-v_2$ & & $x_1$\\ 
         $v_4$ & $v_4$ & $0$ & $v_4$ & $v_4$ & $0$\\
         $w_4$ &  &  & $v_2^2-v_4$ &  & $-v_2x_1+x_3$\\
         $x_1$ &  &  & $0$ &  & $0$\\
         $x_3$ &  &  & $0$ &  & $0$\\ 
         $x_5$ & $x_5$ & $v_2^3-2v_2x_4$ &  $x_5$ & $v_2^2x_1-v_4x_1-v_2x_3+x_5$ & $0$\\ 
         $x_7$ & $x_7$ & $v_2^2v_4-v_2^4$ & $x_7$ & $v_2v_4x_1-v_4x_3+x_7$ & $0$\\
    \end{tabular}$$
    Therefore, $\Lambda W$ with $W=\{v_2,v_4,x_5,x_7\}$ is a minimal Sullivan algebra with the same (co)homology  than $\Lambda V$. Besides, take
    the pairs $\{(x_1,w_2),(x_3,w_4)\}$ obtained when running Algorithm \ref{alg:minimal} on $V$,
   denote $w'_2=w_2+v_2$
    and
    $w'_4=w_4+v_4 + v_2 w_2$ and write $U=\{x_1,x_3\}$, then $dU=\{w'_2,w'_4\}$ and we finally obtain that $\Lambda V$ is isomorphic to $\Lambda W\otimes\Lambda (U\oplus dU) $ being  $\Lambda W$ a minimal Sullivan algebra and $\Lambda (U\oplus dU)$ contractible.
    Besides, the basis of the Sullivan algebra provided by the ``ad hoc" method can be obtained using morphism $g$:
    $$\big\{g(v_2),\; g(v_4),\;g(x_5),\;g(x_7)\big\}$$
    with differential:
    $$d_{\Lambda V}g(v_2)=d_{\Lambda V}g(v_4)=0,$$
    $$d_{\Lambda V}g(x_5)=g(v_2)^3-2g(v_2)g(x_4),\;d_{\Lambda V}g(x_7)=g(v_2)^2g(v_4)-g(v_2)^4.$$
        %$$\big\{g(v_2),\; g(v_4),\;g(x_5),\;g(x_7),\; x_1,\;w_2+v_2,\;x_3,\;w_4+v_4+v_2w_2      \big\}$$
    %with differential:
    %$$d_{\Lambda V}g(v_2)=d_{\Lambda V}g(v_4)=0,$$
    %$$d_{\Lambda V}g(x_5)=g(v_2)^3-2g(v_2)g(x_4),\;d_{\Lambda V}g(x_7)=g(v_2)^2g(v_4)-g(v_2)^4,$$
    %$$ d_{\Lambda V}(x_1)=w_2+v_2,\;d_{\Lambda V}(x_3)=w_4+v_4+v_2w_2.    $$
        \end{ex}
    
    \section{Conclusions and future works}
    
    In this paper, we provide an algorithm that efficiently computes the minimal Sullivan algebra from a Sullivan algebra  not necessarily being minimal.
The  algorithm has been implemented, validated and tested with examples. The  implementation made runs in CoCalc (https://cocalc.com/) and can be downloaded from http://grupo.us.es/\-cimagroup/downloads.htm.

    During the development of this work, we have been aware that M.A. Marco  and V. Manero have been working on a closely related problem to the one we present here. They have implemented in SAGE a program that computes the minimal model of a commutative graded algebra up to a certain degree. The mentioned  work \cite{marco} has been presented in June 2019 in Madrid  in the context of the Int. Conf. on Effective Methods in Algebraic Geometry (MEGA19 website: https://eventos.ucm.es/12097/detail.html). %
Although
they use the ``ad hoc" approach of changing  basis step by step  and we use  chain contractions, we believe it will be interesting in the near future to compare these two points of view.

\end{document}